\documentclass[11pt]{article}
       \usepackage{amsfonts}
       \usepackage{stmaryrd}
       \usepackage{latexsym,amssymb,mathrsfs,fancyhdr}
       \font\tenmsb=msbm10
       \font\sevenmsb=msbm7
       \font\fivemsb=msbm5
       \catcode`\@=11
       \ifx\amstexloaded@\relax\catcode`\@=\active
       \endinput\else\let\amstexloaded@\relax\fi
       \def\spaces@{\space\space\space\space\space}
       \def\spaces@@{\spaces@\spaces@\spaces@\spaces@\spaces@}
       \def\space@.  {\futurelet\space@\relax}
       \space@.   %
       \def\Err@#1{\errhelp\defaulthelp@\errmessage{AmS-TeX error: #1}}
       \def\relaxnext@{\let\next\relax}
       \def\accentfam@{7}
       \def\noaccents@{\def\accentfam@{0}}
       \def\Cal{\relaxnext@\ifmmode\let\next\Cal@\else
       \def\next{\Err@{Use \string\Cal\space only in math mode}}\fi\next}
       \def\Cal@#1{{\Cal@@{#1}}}
       \def\Cal@@#1{\noaccents@\fam\tw@#1}
       \def\Bbb{\relaxnext@\ifmmode\let\next\Bbb@\else
       \def\next{\Err@{Use \string\Bbb\space only in math mode}}\fi\next}
       \def\Bbb@#1{{\Bbb@@{#1}}}
       \def\Bbb@@#1{\noaccents@\fam\msbfam#1}
       \newfam\msbfam
       \textfont\msbfam=\tenmsb
       \scriptfont\msbfam=\sevenmsb
       \scriptscriptfont\msbfam=\fivemsb

\usepackage{dsfont}
\usepackage[german,english]{babel}
\usepackage{amsmath,amssymb}
\usepackage[square, comma, sort&compress, numbers]{natbib}
\usepackage{cases}
\usepackage{mathrsfs}
\usepackage{amsmath}
\usepackage{amsthm}
\usepackage{amsfonts}
\usepackage{amssymb}
\usepackage{latexsym}
\usepackage{fancyhdr}
\usepackage{extarrows}
\usepackage{geometry}
\newtheorem{thm}{Theorem}[section]
\newtheorem{prop}[thm]{Proposition}
\newtheorem{lem}[thm]{Lemma}
\newtheorem{rem}[thm]{Remark}
\newtheorem{iteration lemma}[thm]{iteration Lemma}
\newtheorem{cor}[thm]{Corollary}
\newtheorem{defn}[thm]{Definition}

\newtheorem{eg}[thm]{Example}

\newtheorem*{acknowledgements*}{ACKNOWLEDGEMENTS}

\begin{document}

\setlength{\columnsep}{5pt}
\title{\bf *-DMP elements in $*$-semigroups and $*$-rings}
\author{Yuefeng Gao\footnote{ E-mail: yfgao91@163.com},
 \ Jianlong Chen\footnote{ Corresponding author. E-mail: jlchen@seu.edu.cn},
 \ Yuanyuan Ke\footnote{ E-mail: keyy086@126.com} \\
Department of  Mathematics, Southeast University, \\ Nanjing 210096,  China}
     \date{}

\maketitle
\begin{quote}
{\textbf{Abstract: }\small In this paper, we investigate *-DMP elements in $*$-semigroups and $*$-rings. The notion of *-DMP element was introduced by Patr\'{i}cio in 2004. An element $a$ is *-DMP if there exists a positive integer $m$ such that $a^{m}$ is EP. We first characterize *-DMP elements in terms of the \{1,3\}-inverse, Drazin inverse and pseudo core inverse, respectively.
 Then, we give the pseudo core decomposition utilizing the pseudo core inverse, which extends the core-EP decomposition introduced by Wang for matrices to an arbitrary $*$-ring; and this decomposition turns to be a  useful tool to characterize *-DMP elements. Further, we extend Wang's core-EP order from matrices to $*$-rings and use it to investigate *-DMP elements. Finally, we give necessary and sufficient conditions for two elements $a,~b$ in $*$-rings to have $aa^{\scriptsize\textcircled{\tiny D}}=bb^{\scriptsize\textcircled{\tiny D}}$, which contribute to investigate *-DMP elements.


}
\end{quote}

\section{ Introduction }
Let $S$ and $R$ denote a semigroup and a ring with unit 1, respectively.

An element $a\in S$ is Drazin invertible \cite{D1958} if there exists the unique element $a^D\in S$ such that
$$a^ma^Da=a^m~\text{for some positive integer}~m,~~a^Daa^D=a^D~\text{and}~~aa^D=a^Da.$$
The smallest positive integer $m$ satisfying above equations is called the Drazin index of $a$. We denote by $a^{D_m}$ the Drazin inverse of index $m$ of $a$. If the Drazin index of $a$ equals one, then the Drazin inverse of $a$ is called the group inverse of $a$ and is denoted by $a^{\#}$.

$S$ is called a $*$-semigroup if $S$ is a semigroup with involution $*$. $R$ is called a $*$-ring if $R$ is a ring with involution $*$. In the following, unless otherwise indicated, $S$ and $R$ denote a $*$-semigroup and a $*$-ring, respectively.

An element $a\in S$ is Moore-Penrose invertible, if there exists $x\in S$ such that
$$(1)~axa=a,~~(2)~xax=x,~~(3)~(ax)^{*}=ax~~\text{and}~~(4)~(xa)^{*}=xa.$$
If such $x$ exists, then it is unique, denoted by $a^{\dag}$. $x$ satisfying equations $(1)$ and~$(3)$ is called a $\{1,3\}$-inverse of $a$, denoted by $a^{(1,3)}$. Such $\{1,3\}$-inverse of $a$ is not unique if it exists. We use $a\{1,3\}$, $S^{\{1,3\}}$ to denote the set of all the $\{1,3\}$-inverses of $a$ and the set of all the $\{1,3\}$-invertible elements in $S$, respectively.

An element $a\in S$ is symmetric if $a^*=a$. $a\in S$ is EP if $a^{\#}$ and $a^{\dag}$ exists with $a^{\#}=a^{\dag}$. $a\in S$ is *-DMP with index $m$ if $m$ is the smallest positive integer such that $(a^{m})^{\#}$ and $(a^{m})^{\dag}$ exist with $(a^{m})^{\#}=(a^{m})^{\dag}$ \cite{D2004}. In other words, $a\in S$ is *-DMP with index $m$ if $m$ is the smallest positive integer such that $a^{m}$ is EP. We call $a\in S$ *-DMP if there exists a positive integer $m$ such that $a^{m}$ is EP.

Baksalary and Trenkler \cite{D2010} introduced the notion of core inverse for complex matrices in 2010. Then, Raki\'{c}, Din\v{c}i\'{c} and Djordjevi\'{c} \cite{DDD2014} generalized this notion to an arbitrary $*$-ring. Later, Xu, Chen and Zhang \cite{X2016} characterized the core invertible elements in $*$-rings in terms of three equations. The core inverse of $a$ is the unique solution to equations $$xa^2=a,~~ax^2=x,~~(ax)^{*}=ax.$$

The pseudo core inverse \cite{A2016} of $a\in S$ is the unique element $a^{\scriptsize\textcircled{\tiny D}}\in S$ satisfying the following three equations~$$a^{\scriptsize\textcircled{\tiny D}}a^{m+1}=a^m~\text{for some positive integer}~m,~a(a^{\scriptsize\textcircled{\tiny D}})^2=a^{\scriptsize\textcircled{\tiny D}}~\text{and}~(aa^{\scriptsize\textcircled{\tiny D}})^{*}=aa^{\scriptsize\textcircled{\tiny D}}.$$
The smallest  positive integer $m$ satisfying above equations is called
the pseudo core index of $a$, denoted by $I(a)$. The pseudo core inverse is a kind of outer inverse.
If the pseudo core index equals one, then the pseudo core inverse of $a$ is the core inverse of $a$, denoted by $a^{\tiny\textcircled{\tiny \#}}$. The pseudo core inverse extends core-EP inverse \cite{2014D} from matrices to $*$-semigroups (see \cite{A2016}).

Dually, the dual pseudo core inverse \cite{A2016} of $a\in S$ is the unique element $a_{\scriptsize\textcircled{\tiny D}}\in S$ satisfying the following three equations~$$a^{m+1}a_{\scriptsize\textcircled{\tiny D}}=a^m~\text{for some positive integer}~m,~(a_{\scriptsize\textcircled{\tiny D}})^2a=a_{\scriptsize\textcircled{\tiny D}}~\text{and}~(a_{\scriptsize\textcircled{\tiny D}}a)^{*}=a_{\scriptsize\textcircled{\tiny D}}a.$$
The smallest  positive integer $m$ satisfying above equations is called
the dual pseudo core index of $a$.
We denote by $a^{\scriptsize\textcircled{\tiny D}_m}$ and $a_{\scriptsize\textcircled{\tiny D}_m}$ the pseudo core inverse and dual pseudo core inverse of index $m$ of $a$, respectively. Note that $(a^*)^{\scriptsize\textcircled{\tiny D}_m}$ exists if and only if $a_{\scriptsize\textcircled{\tiny D}_m}$ exists with $(a^*)^{\scriptsize\textcircled{\tiny D}_m}=(a_{\scriptsize\textcircled{\tiny D}_m})^*$.

Wang \cite{2016} introduced the core-EP decomposition and core-EP order for complex matrices by applying the core-EP inverse.

In this paper,  we mainly  characterize *-DMP elements in terms of the pseudo core inverse.
In Section 2, we first give several characterizations of *-DMP elements in terms of the \{1,3\}-inverse, Drazin inverse and pseudo core inverse respectively.
Then, we characterize *-DMP elements in terms of equations and annihilators. At last, we consider conditions for the sum (resp. product) of two *-DMP elements to be *-DMP.
In section 3, we extend the core-EP decomposition from complex matrices to an arbitrary $*$-ring and name it pseudo core decomposition, since it seems no connection with EP in $*$-rings, but closely related with the pseudo core inverse. As applications,  we use it to characterize *-DMP elements.
In Section 4, we introduce the pseudo core order, which extends the core-EP order from complex matrices to an arbitrary $*$-ring. Then, we use this pre-order to investigate *-DMP elements.
In final section, we aim to give equivalent conditions for $aa^{\scriptsize\textcircled{\tiny D}}=bb^{\scriptsize\textcircled{\tiny D}}$ in $*$-rings, which contribute to investigate *-DMP elements.

\section{Characterizations of *-DMP elements}
EP elements are widely investigated in $*$-semigroups and $*$-rings ( see \cite{N2006}, \cite{W2012}, \cite{D1976}, \cite{DD1976}, \cite{M2009}, \cite{2009}, \cite{D2004}).
In this section,  we give several characterizations of *-DMP elements. We begin with some auxiliary lemmas.
\begin{lem}\label{4} \emph{\cite{A2016}}~Let $a\in S$. Then we have the following facts:\\
$(1)$ $a^{\scriptsize\textcircled{\tiny D}_m}$ exists~if and only if~$a^{D_m}$ exists and $a^m\in S^{\{1,3\}}$. In this case $a^{\scriptsize\textcircled{\tiny D}_m}=a^{D_m}a^m(a^m)^{(1,3)}$. \\
$(2)$ $a^{\scriptsize\textcircled{\tiny D}_m}$ and $a_{\scriptsize\textcircled{\tiny D}_m}$ exist~if and only if~$a^{D_m}$~and~$(a^m)^{\dag}$ exist.
In this case, $a^{\scriptsize\textcircled{\tiny D}_m}=a^{D_m}a^m(a^m)^{\dag}$ and $a_{\scriptsize\textcircled{\tiny D}_m}=(a^m)^{\dag}a^ma^{D_m}$.
\end{lem}

\begin{lem}\emph{\cite{D2002}},\emph{\cite{D2004}} \label{2}~Let $a\in S$. Then the following conditions are equivalent:\\
$(1)$~$a$ is *-DMP with index $m$;\\
$(2)$~$a^{D_m}$ exists and $aa^{D_m}$ is symmetric.
\end{lem}

\begin{lem}\label{1}~Let $a\in S$. Then the following are equivalent:\\
$(1)$~$a$ is *-DMP with index $m$;\\
$(2)$~$a^{D_m}$ and $(a^m)^{\dag}$ exist with $(a^{D_m})^m=(a^m)^{\dag}$;\\
$(3)$~$a^{\scriptsize\textcircled{\tiny D}_m}$ exists with $a^{\scriptsize\textcircled{\tiny D}_m}=a^{D_m}$;\\
$(4)$~$a^{\scriptsize\textcircled{\tiny D}_m}$ and $(a^m)^{\dag}$ exist with $(a^{\scriptsize\textcircled{\tiny D}_m})^m=(a^m)^{\dag}$.
\end{lem}
\begin{proof}~$(1)\Rightarrow(2)$. Suppose $a$ is *-DMP with index $m$, then $m$ is the smallest positive integer such that $(a^m)^{\dag}=(a^m)^{\#}$. So, $a^{D_m}$ exists with $(a^{D_m})^m=(a^m)^{\#}=(a^m)^{\dag}$.\\
$(2)\Rightarrow(3)$. Suppose $a^{D_m}$ and $(a^m)^{\dag}$ exist with $(a^{D_m})^m=(a^m)^{\dag}$. By Lemma \ref{4}, $a^{\scriptsize\textcircled{\tiny D}_m}$ exists with $a^{\scriptsize\textcircled{\tiny D}_m}=a^{D_m}a^m(a^m)^{\dag}=a^{D_m}a^m(a^{D_m})^m=a^{D_m}$.\\
$(3)\Rightarrow(4)$. Applying Lemma \ref{4}, $a^{\scriptsize\textcircled{\tiny D}_m}$ exists if and only if $a^{D_m}$ exists and $a^m\in S^{\{1,3\}}$, in which case, $a^{\scriptsize\textcircled{\tiny D}_m}=a^{D_m}a^m(a^m)^{(1,3)}$. From $a^{\scriptsize\textcircled{\tiny D}_m}=a^{D_m}$, it follows that $a^{D_m}a^m(a^m)^{(1,3)}=a^{D_m}$. Then, $aa^{D_m}=a^m(a^m)^{(1,3)}$. So, $(a^m)^{\dag}$ exists with $(a^m)^{\dag}=(a^{D_m})^m=(a^{\scriptsize\textcircled{\tiny D}_m})^m$.\\
$(4)\Rightarrow(1)$. Since  $(a^{D_m})^ma^m(a^m)^{(1,3)}=(a^{\scriptsize\textcircled{\tiny D}_m})^m=(a^m)^{\dag}$, then $aa^{D_m}=(a^m)^{\dag}a^m$. Therefore $aa^{D_m}$ is symmetric.
Hence $a$ is *-DMP with index $m$ by Lemma \ref{2}.
\end{proof}
\vspace{2mm}

The following result characterizes *-DMP elements in terms of  $\{1,3\}$-inverses.
\begin{thm}\label{3}~Let $a\in S$. Then $a$ is *-DMP with index $m$ if and only if $m$ is the smallest positive integer such that $a^m\in S^{\{1,3\}}$ and one of the following equivalent conditions holds:\\
$(1)$~$a(a^m)^{(1,3)}=(a^m)^{(1,3)}a$ for some $(a^m)^{(1,3)}\in a^m\{1,3\}$;\\
$(2)$~$a^m(a^m)^{(1,3)}=(a^m)^{(1,3)}a^m$ for some $(a^m)^{(1,3)}\in a^m\{1,3\}$.
\end{thm}

\begin{proof}~If $a$ is *-DMP with index $m$, then $m$ is the smallest positive integer such that $(a^m)^{\dag}$ and $(a^m)^{\#}$ exist with $(a^m)^{\dag}=(a^m)^{\#}$. So we may regard $(a^m)^{\#}$ as one of the $\{1,3\}$-inverses of $a^m$. Therefore (1) holds (see \cite[Theorem 1]{D1958}).

Conversely, we take $(a^m)^{(1,3)}\in a^m\{1,3\}$.\\
$(1)\Rightarrow(2)$ is obvious.\\
$(2)$.~Equality $a^m(a^m)^{(1,3)}=(a^m)^{(1,3)}a^m$ yields that $(a^m)^{\dag}=(a^m)^{(1,3)}a^m(a^m)^{(1,3)}=(a^m)^{\#}$. So $m$ is the smallest positive integer such that $(a^m)^{\dag}=(a^m)^{\#}$. Hence $a$ is $*$-DMP with index $m$.
\end{proof}

\begin{cor}~Let $a\in S$. Then $a$ is EP if and only if $a\in S^{\{1,3\}}$ and $aa^{(1,3)}=a^{(1,3)}a$ for some $a^{(1,3)}\in a\{1,3\}$.
\end{cor}
\vspace{4mm}

In \cite[Theorem 7.3]{D2002}, Koliha and Patr\'{i}cio characterized EP elements by using the group inverse. Similarly, we
characterize *-DMP elements in terms of the Drazin inverse.
\begin{thm}\label{18}~Let $a\in S$. Then $a$ is *-DMP with index $m$ if and only if $a^{D_m}$ exists and one of the following equivalent conditions holds:\\
$(1)$~$a^{D_m}=a^{D_m}(aa^{D_m})^*$;\\
$(2)$~$a^{D_m}=(a^{D_m}a)^*a^{D_m}$.\\
If $S$ is a $*$-ring, then $(1)$-$(2)$ are equivalent to\\
$(3)$~$a^{D_m}(1-aa^{D_m})^*=(1-aa^{D_m})(a^{D_m})^*$.
\end{thm}
\begin{proof}~If $a$ is *-DMP with index $m$, then $a^{D_m}$ exists and
 $aa^{D_m}$ is symmetric by Lemma~\ref{2}. It is not difficult to verify that conditions (1)-(3) hold.

 Conversely, we assume that $a^{D_m}$ exists.\\
 $(1)\Rightarrow(3)$.~Since $a^{D_m}=a^{D_m}(aa^{D_m})^*$, we have $$a^{D_m}(1-aa^{D_m})^*=a^{D_m}(aa^{D_m})^*(1-aa^{D_m})^*=a^{D_m}((1-aa^{D_m})aa^{D_m})^*=0.$$
 Therefore $a^{D_m}(1-aa^{D_m})^*=0=(1-aa^{D_m})(a^{D_m})^*$.\\
 $(2)\Rightarrow(3)$~is analogous to $(1)\Rightarrow(3)$.

 Finally, we will prove $a$ is *-DMP with index $m$ under the assumption that $a^{D_m}$ exists with $a^{D_m}(1-aa^{D_m})^*=(1-aa^{D_m})(a^{D_m})^*$. From $a^{D_m}(1-a^*(a^{D_m})^*)=(1-a^{D_m}a)(a^{D_m})^*$, we get $(a^{D_m})^*=a^{D_m}(1-a^*(a^{D_m})^*+a(a^{D_m})^*)$. Post-multiply this equality by $(a^{D_m})^*(a^2)^*$, then we have $aa^{D_m}=aa^{D_m}(aa^{D_m})^*$. So $aa^{D_m}$ is symmetric. Applying Lemma \ref{2}, $a$ is *-DMP with index $m$.
\end{proof}
\vspace{4mm}




Let us recall that $a\in S$ is normal if $aa^*=a^*a$. It is known that an element $a\in S$ is EP may not imply it is normal~(such as, take $S=\mathbb{R}^{2\times 2}$ with transpose as involution, $a=(\begin{smallmatrix}
     1&1\\
     0&1
   \end{smallmatrix})$. Then $a$ is EP since $aa^{\dag}=a^{\dag}a=1$,
                               but $aa^*=(\begin{smallmatrix}
                                         2&1\\
                                         1&1
                                        \end{smallmatrix})\neq (\begin{smallmatrix}
                                                            1&1\\
                                                            1&2
                                                          \end{smallmatrix})=a^*a$); $a$ is normal may not imply it is EP (such as, take $S=\mathbb{C}^{2\times 2}$ with transpose as involution,
$a=(\begin{smallmatrix}
i&1\\
-1&i
\end{smallmatrix})$. Then $aa^*=a^*a=0$, i.e., $a$ is normal. But $a$ is not Moore-Penrose invertible and hence $a$ is not EP).
So we may be of interest to know when $a$ is both EP and normal. Here we give a more extensive case.
\begin{thm}~Let $a\in S$. Then the following are equivalent:\\
$(1)$~$a$ is *-DMP with index $m$ and $a(a^*)^m=(a^*)^ma$;\\
$(2)$~$m$ is the smallest positive integer such that $(a^m)^{\dag}$ exists and $a(a^*)^m=(a^*)^ma$;\\
$(3)$~$a^{D_m}$ exists and $(a^m)^*=ua=au$ for some group invertible element $u\in S$.
\end{thm}

\begin{proof}~$(1)\Rightarrow(2)$ is clear.\\
$(2)\Rightarrow(1)$. The equality $a^m(a^m)^*=(a^m)^*a^m$ ensures that $a^m(a^m)^{\dag}=(a^m)^{\dag}a^m$ (see \cite[Theorem 5]{D1992}). So $a$ is *-DMP with index $m$ by Theorem \ref{3}.\\
$(1)\Rightarrow(3)$. Since $a$ is *-DMP with index $m$, then $a^{D_m}$ exists and $aa^{D_m}$ is symmetric by Lemma~\ref{2}. So,
\begin{equation*}
\begin{aligned}
(a^m)^*&=(a^ma^{D_m}a)^*=aa^{D_m}(a^m)^*,~\mbox{and}\\
(a^m)^*&=(aa^{D_m}a^m)^*=(a^m)^*aa^{D_m}.
\end{aligned}
\end{equation*}
Since $a^{D_m}$ exists and $(a^m)^*a=a(a^m)^*$, then we obtain $a^{D_m}(a^m)^*=(a^m)^*a^{D_m}$~(see~\cite[Theorem 1]{D1958}).
Take $u=a^{D_m}(a^m)^*$, then $au=ua=(a^m)^*$. In what follows, we show $u^{\#}=a((a^{D_m})^m)^*$.
In fact,
\begin{equation*}
\begin{aligned}
(\mathrm i)~
ua((a^{D_m})^m)^*u&=a^{D_m}(a^m)^*a((a^{D_m})^m)^*a^{D_m}(a^m)^*
=(a^m)^*aa^{D_m}((a^{D_m})^m)^*a^{D_m}(a^m)^*~~~~\\
&=(a^m)^*((a^{D_m})^m)^*a^{D_m}(a^m)^*=(aa^{D_m})^*a^{D_m}(a^m)^*
=a^{D_m}(a^m)^*=u;
\end{aligned}
\end{equation*}
\begin{equation*}
\begin{aligned}
(\mathrm {ii})~a((a^{D_m})^m)^*ua((a^{D_m})^m)^*&=a((a^{D_m})^m)^*a^{D_m}(a^m)^*a((a^{D_m})^m)^*\\
&=a((a^{D_m})^m)^*(a^m)^*a^{D_m}a((a^{D_m})^m)^*\\
&=a(aa^{D_m})^*a^{D_m}a((a^{D_m})^m)^*=a(aa^{D_m})^*((a^{D_m})^m)^*~~~~~~~~~~~~~\\
&=a((a^{D_m})^m)^*;
\end{aligned}
\end{equation*}
\begin{equation*}
\begin{aligned}
(\mathrm{iii})~
a((a^{D_m})^m)^*u&=a((a^{D_m})^m)^*a^{D_m}(a^m)^*=a((a^{D_m})^m)^*(a^m)^*a^{D_m}=a(aa^{D_m})^*a^{D_m}~~~~~~~~\\
&=aa^{D_m}~\text{and}\\
~ua((a^{D_m})^m)^*&=a^{D_m}(a^m)^*a((a^{D_m})^m)^*=a^{D_m}a(a^m)^*((a^{D_m})^m)^*=aa^{D_m},\\
\text{so,~}a((a^{D_m})^m&)^*u=ua((a^{D_m})^m)^*.\\
\end{aligned}
\end{equation*}
Hence~$u^{\#}=a((a^{D_m})^m)^*$. \\
$(3)\Rightarrow(1)$.~Since $u^{\#}$ and $a^{D_m}$ exist with $au=ua$, then $au^{\#}=u^{\#}a$ and $(ua)^D=u^{\#}a^{D_m}$. \\
So,~
$(aa^{D_m})^*=(a^m(a^{D_m})^m)^*=((a^m)^{D}a^m)^*=(a^m)^*((a^m)^{*})^D
=ua(ua)^D=uau^{\#}a^{D_m}\\
\indent\indent\indent\quad=uu^{\#}aa^D$.\\Therefore~
$(aa^{D_m})^*aa^{D_m}=uu^{\#}aa^{D_m}=(aa^{D_m})^*$. That is, $aa^{D_m}$ is symmetric.\\
We thus have $a$ is *-DMP with index $m$.
\end{proof}

\begin{cor}~Let $a\in S$. Then the following are equivalent:\\
$(1)$~$a$ is EP and normal;\\
$(2)$~$a^{\dag}$ exists and $a$ is normal;\\
$(3)$~$a^{\#}$ exists and $a^*=ua=au$ for some group invertible element $u\in S$.
\end{cor}
\vspace{4mm}
In what follows, *-DMP elements are characterized in terms of the pseudo core inverse and dual pseudo core inverse.

\begin{thm}\label{22}~Let $a\in S$. Then the following are equivalent:\\
$(1)$~$a$ is *-DMP with index $m$;\\
$(2)$~$a^{\scriptsize\textcircled{\tiny D}_m}$ and $a_{\scriptsize\textcircled{\tiny D}_m}$ exist with
$a^{\scriptsize\textcircled{\tiny D}_m}=a_{\scriptsize\textcircled{\tiny D}_m}$;\\
$(3)$~$a^{\scriptsize\textcircled{\tiny D}_m}$ and $a_{\scriptsize\textcircled{\tiny D}_m}$ exist with
$aa^{\scriptsize\textcircled{\tiny D}_m}=a_{\scriptsize\textcircled{\tiny D}_m}a$.
\end{thm}
\begin{proof}~$(1)\Rightarrow(2)$, (3). If $a$ is *-DMP with index $m$, then by Lemma \ref{1}, $a^{D_m}$ and $(a^m)^{\dag}$ exist with $(a^m)^{\dag}=(a^{D_m})^m$. Hence $a^{\scriptsize\textcircled{\tiny D}_m}$ and $a_{\scriptsize\textcircled{\tiny D}_m}$ exist by Lemma \ref{4} (2). It is not difficult to verify that $a_{\scriptsize\textcircled{\tiny D}_m}=a^{\scriptsize\textcircled{\tiny D}_m}$ and $aa^{\scriptsize\textcircled{\tiny D}_m}=a_{\scriptsize\textcircled{\tiny D}_m}a$.

$(2)\Rightarrow(1)$. If $a^{\scriptsize\textcircled{\tiny D}_m}$ and $a_{\scriptsize\textcircled{\tiny D}_m}$ exist, then $a^{D_m}$ and $(a^m)^{\dag}$ exist with $a^{\scriptsize\textcircled{\tiny D}_m}=a^{D_m}a^m(a^m)^{\dag}$, $a_{\scriptsize\textcircled{\tiny D}_m}=(a^m)^{\dag}a^ma^{D_m}$.
Equality $a_{\scriptsize\textcircled{\tiny D}_m}=a^{\scriptsize\textcircled{\tiny D}_m}$ would imply that $a^{D_m}a^m(a^m)^{\dag}=(a^m)^{\dag}a^ma^{D_m}$.  Post-multiply this equality by $a^{m+1}(a^{D_m})^m$, then we obtain $aa^{D_m}=(a^m)^{\dag}a^m$. So $aa^{D_m}$ is symmetric. According to Lemma \ref{2}, $a$ is *-DMP with index $m$.

$(3)\Rightarrow(1)$. By the hypothesis, we have $aa^{D_m}a^{m}(a^m)^{\dag}=(a^m)^{\dag}a^{m}a^{D_m}a$. That is, $a^{m}(a^m)^{\dag}=(a^m)^{\dag}a^{m}$. So $aa^{D_m}=a^m(a^{D_m})^m=a^{m}(a^m)^{\dag}a^m(a^{D_m})^m=(a^m)^{\dag}a^{m}a^m(a^{D_m})^m=(a^m)^{\dag}a^{m}$. Therefore $aa^{D_m}$ is symmetric. Hence $a$ is *-DMP with index $m$.
\end{proof}
\vspace{4mm}

The following result characterizes *-DMP elements merely in terms of the pseudo core inverse.
\begin{thm}\label{5}~Let $a\in S$. Then $a$ is *-DMP with index $m$ if and only if $a^{\scriptsize\textcircled{\tiny D}_m}$ exists and one of the following equivalent conditions holds:\\
$(1)$~$aa^{\scriptsize\textcircled{\tiny D}_m}=a^{\scriptsize\textcircled{\tiny D}_m}a$;\\
$(2)$~$a^{D_m}a^{\scriptsize\textcircled{\tiny D}_m}=a^{\scriptsize\textcircled{\tiny D}_m}a^{D_m}$;\\
$(3)$~$a^{\scriptsize\textcircled{\tiny D}_m}=(a^m)^{(1,3)}a^ma^{D_m}$ for some $(a^m)^{(1,3)}\in a^m\{1,3\}$;\\
$(4)$~$a^{m+1}a^{\scriptsize\textcircled{\tiny D}_m}=a^m$;\\
$(5)$~$(a^{\scriptsize\textcircled{\tiny D}_m})^2a=a^{\scriptsize\textcircled{\tiny D}_m}$;\\
$(6)$~$a^{\scriptsize\textcircled{\tiny D}_m}a$ is symmetric;\\
$(7)$~$aa^{\scriptsize\textcircled{\tiny D}_m}$ commutes with $a^{\scriptsize\textcircled{\tiny D}_m}a$.
\end{thm}
\begin{proof}~If $a$ is *-DMP with index $m$, then $(a^{D_m})^m=(a^m)^{\dag},~a^{\scriptsize\textcircled{\tiny D}_m}=a^{D_m}$ by Lemma \ref{1} and $aa^{D_m}$ is symmetric by Lemma \ref{2}. So (1)-(7) hold.

Conversely, we assume that $a^{\scriptsize\textcircled{\tiny D}_m}$ exists.\\
(1).~By the definition of the pseudo core inverse, we have $a^{\scriptsize\textcircled{\tiny D}_m}a^{m+1}=a^m$, and we also have $a^{\scriptsize\textcircled{\tiny D}_m}aa^{\scriptsize\textcircled{\tiny D}_m}=a^{\scriptsize\textcircled{\tiny D}_m}$ by calculation.
The equalities $aa^{\scriptsize\textcircled{\tiny D}_m}=a^{\scriptsize\textcircled{\tiny D}_m}a$, $a^{\scriptsize\textcircled{\tiny D}_m}aa^{\scriptsize\textcircled{\tiny D}_m}=a^{\scriptsize\textcircled{\tiny D}_m}$ and $a^{\scriptsize\textcircled{\tiny D}_m}a^{m+1}=a^m$ yield that  $a^{D_m}=a^{\scriptsize\textcircled{\tiny D}_m}$. Therefore $a$ is *-DMP with index $m$ by Lemma \ref{1}.\\
(2).~Since $a^{D_m}a^{\scriptsize\textcircled{\tiny D}_m}=a^{\scriptsize\textcircled{\tiny D}_m}a^{D_m}$, then
$(a^{D_m})^{\#}a^{\scriptsize\textcircled{\tiny D}_m}=a^{\scriptsize\textcircled{\tiny D}_m}(a^{D_m})^{\#}$ (see \cite[Theorem 1]{D1958}).
Namely,
$$a^2a^{D_m}a^{\scriptsize\textcircled{\tiny D}_m}=a^{\scriptsize\textcircled{\tiny D}_m}a^2a^{D_m}.$$
So $aa^{\scriptsize\textcircled{\tiny D}_m}
=a^m(a^{\scriptsize\textcircled{\tiny D}_m})^m
=aa^{D_m}a^m(a^{\scriptsize\textcircled{\tiny D}_m})^m
=aa^{D_m}aa^{\scriptsize\textcircled{\tiny D}_m}
=a^2a^{D_m}a^{\scriptsize\textcircled{\tiny D}_m}=a^{\scriptsize\textcircled{\tiny D}_m}a^2a^{D_m}\\
\indent\indent~~\!~=a^{\scriptsize\textcircled{\tiny D}_m}a^{m+1}(a^{D_m})^m
=a^m(a^{D_m})^m=aa^{D_m}$. \\
Therefore $aa^{D_m}$ is symmetric. Hence $a$ is *-DMP with index $m$ by Lemma \ref{2}. \\
(3).~Since $a^{\scriptsize\textcircled{\tiny D}_m}$ exists, then by Lemma \ref{4} (1), $a^{D_m}$ and $(a^m)^{(1,3)}$ exist. From equality (3) and $a^{\scriptsize\textcircled{\tiny D}_m}=a^{D_m}a^m(a^m)^{(1,3)}$, it follows that $a^{D_m}a^m(a^m)^{(1,3)}=(a^m)^{(1,3)}a^ma^{D_m}$. Pre-multiply this equality by $(a^{D_m})^{m-1}a^m$, then we get $$a^m(a^m)^{(1,3)}=aa^{D_m}.$$
So $aa^{D_m}$ is symmetric. Hence $a$ is *-DMP with index $m$ by Lemma~\ref{2}. \\
(4).~The equalities $a^{m+1}a^{\scriptsize\textcircled{\tiny D}_m}=a^m$ and $a^{\scriptsize\textcircled{\tiny D}_m}a^{m+1}=a^m$ yield that $a$ is strongly $\pi$-regular and $a^{D_m}=a^{m}(a^{\scriptsize\textcircled{\tiny D}_m})^{m+1}=a^{\scriptsize\textcircled{\tiny D}_m}$~(see \cite[Theorem 4]{D1958}). So $a$ is *-DMP with index $m$ by Lemma \ref{1}.\\
$(5)\Rightarrow (1)$. Pre-multiply (5) by $a$, then we get $a(a^{\scriptsize\textcircled{\tiny D}_m})^2a=aa^{\scriptsize\textcircled{\tiny D}_m}$. That is, $a^{\scriptsize\textcircled{\tiny D}_m}a=aa^{\scriptsize\textcircled{\tiny D}_m}$. \\
$(6)\Rightarrow (1)$. Pre-multiply $(a^{\scriptsize\textcircled{\tiny D}_m}a)^*=a^{\scriptsize\textcircled{\tiny D}_m}a$ by $aa^{\scriptsize\textcircled{\tiny D}_m}$, then we obtain $$aa^{\scriptsize\textcircled{\tiny D}_m}(a^{\scriptsize\textcircled{\tiny D}_m}a)^*=aa^{\scriptsize\textcircled{\tiny D}_m}a^{\scriptsize\textcircled{\tiny D}_m}a=a^{\scriptsize\textcircled{\tiny D}_m}a.$$
So $aa^{\scriptsize\textcircled{\tiny D}_m}
=a^m(a^{\scriptsize\textcircled{\tiny D}_m})^m
=(a^m(a^{\scriptsize\textcircled{\tiny D}_m})^m)^*
=(a^{\scriptsize\textcircled{\tiny D}_m}a^{m+1}(a^{\scriptsize\textcircled{\tiny D}_m})^m)^*
=(a^{\scriptsize\textcircled{\tiny D}_m}aaa^{\scriptsize\textcircled{\tiny D}_m})^*\\
\indent\indent\quad\!=(aa^{\scriptsize\textcircled{\tiny D}_m})^*(a^{\scriptsize\textcircled{\tiny D}_m}a)^*
=aa^{\scriptsize\textcircled{\tiny D}_m}(a^{\scriptsize\textcircled{\tiny D}_m}a)^*
=a^{\scriptsize\textcircled{\tiny D}_m}a$.\\
$(7)\Rightarrow (1)$. From $aa^{\scriptsize\textcircled{\tiny D}_m}(a^{\scriptsize\textcircled{\tiny D}_m}a)=(a^{\scriptsize\textcircled{\tiny D}_m}a)aa^{\scriptsize\textcircled{\tiny D}_m}$,  $aa^{\scriptsize\textcircled{\tiny D}_m}(a^{\scriptsize\textcircled{\tiny D}_m}a)=a^{\scriptsize\textcircled{\tiny D}_m}a$ and $(a^{\scriptsize\textcircled{\tiny D}_m}a)aa^{\scriptsize\textcircled{\tiny D}_m}=a^{\scriptsize\textcircled{\tiny D}_m}a^{m+1}(a^{\scriptsize\textcircled{\tiny D}_m})^m=aa^{\scriptsize\textcircled{\tiny D}_m}$, it follows that $aa^{\scriptsize\textcircled{\tiny D}_m}=a^{\scriptsize\textcircled{\tiny D}_m}a$. 

\end{proof}
\vspace{4mm}

In \cite{DD2016}, Xu and Chen characterized EP elements in terms of equations. Similarly, we utilize equations to characterize *-DMP elements.
\begin{thm}\label{21}~Let $a\in S$. Then the following are equivalent:\\
$(1)$~$a$ is *-DMP with index $m$;\\
$(2)$~$m$ is the smallest positive integer such that $xa^{m+1}=a^m,~ax^2=x~\text{and}~(x^ma^m)^*=x^ma^m$ for some $x\in S$;\\
$(3)$~$m$ is the smallest positive integer such that $xa^{m+1}=a^m,~ax=xa~\text{and}~(x^ma^m)^*=x^ma^m$ for some $x\in S$.
\end{thm}

\begin{proof}~$(1)\Rightarrow(2)$, $(3)$.~Suppose $a$ is *-DMP with index $m$, then $a^{D_m}$ exists and $a^{D_m}a$ is symmetric by Lemma~\ref{2}. Take $x=a^{D_m}$, then (2) and (3) hold.\\
$(2)\Rightarrow(1)$. From $xa^{m+1}=a^m$ and $a^m=xa^{m+1}=a^{m+1}x^{m+1}a^{m}$, it follows that $a$ is strongly $\pi$-regular and $a^{D_m}=x^{m+1}a^m$. So $aa^{D_m}=ax^{m+1}a^m=x^ma^m$. Therefore $a^{D_m}$ exists and $aa^{D_m}$ is symmetric. Hence $a$ is *-DMP with index $m$ by Lemma~\ref{2}.\\
$(3)\Rightarrow(1)$. Equalities $xa^{m+1}=a^m$ and $a^m=a^{m+1}x$ yield that $a^{D_m}=x^{m+1}a^m$. So $a^{D_m}a=x^{m+1}a^{m+1}=x^ma^m$. Therefore $a^{D_m}$ exists and $aa^{D_m}$ is symmetric. Hence $a$ is *-DMP with index $m$.
\end{proof}
\vspace{4mm}

Let $S^0$ denote a $*$-semigroup with zero element $0$.
The left annihilator of $a\in S^0$ is denoted by $^\circ a$ and is defined by $^\circ a=\{x\in S^0: xa=0\}$.
The following result characterizes *-DMP elements in $S^0$ in terms of left annihilators.
We begin with an auxiliary lemma.
\begin{lem}\label{7} \emph{\cite{A2016}}~Let $a,~x\in S^0$. Then $a^{\scriptsize\textcircled{\tiny D}_m}=x$ if and only if $m$ is the smallest positive integer  such that one of the following equivalent conditions holds:\\
$(1)$~$xax=x$ and $xS^0=x^*S^0=a^mS^0$;\\
$(2)$~$xax=x$, $^\circ x=\!\!~^\circ(a^m)$ and $^\circ(x^*)\subseteq\!\!~^\circ(a^m)$.
\end{lem}
\begin{thm}~Let $a\in S^0$. Then $a$ is *-DMP with index $m$ if and only if $m$ is the smallest positive integer such that one of the following equivalent conditions holds:\\
$(1)$~$xax=x,~xS^0=x^*S^0=a^mS^0$ and $x^mS^0=(a^m)^*S^0$ for some $x\in S^0$;\\
$(2)$~$xax=x$, $^\circ{x}=\!\!~^\circ(a^m)$, $^\circ(x^*)\subseteq\!\!~^\circ(a^m)$ and $^\circ(a^m)^* \subseteq\!\!~^\circ (x^m)$ for some $x\in S^0$.
\end{thm}

\begin{proof}~Suppose $a$ is *-DMP with index $m$. Then $a^{\scriptsize\textcircled{\tiny D}_m},~(a^m)^{\dag}$ exist with $(a^{\scriptsize\textcircled{\tiny D}_m})^m=(a^m)^{\dag}$ by
Lemma~\ref{1}.  Take $x=a^{\scriptsize\textcircled{\tiny D}_m}$, then $xax=x,~xS^0=x^*S^0=a^mS^0$ by Lemma~\ref{7}. Further, from $x^m=(a^m)^{\dag}$, it follows that
$x^m=(x^ma^m)^*x^m=(a^m)^*(x^m)^*x^m\in (a^m)^*S^0$ and $(a^m)^*=(a^mx^ma^m)^*=x^ma^m(a^m)^*\in x^mS^0$. Hence (1) holds.\\
\noindent $(1)\Rightarrow(2)$ is clear.\\
(2). From $xax=x$, $^\circ{x}=\!\!~^\circ(a^m)$ and $^\circ(x^*)\subseteq\!\!~^\circ(a^m)$, it follows that $a^{\scriptsize\textcircled{\tiny D}_m}=x$ by Lemma \ref{7}. Then $1-(x^ma^m)^*\in\!\!~ ^\circ (a^m)^*\subseteq\!\!~^\circ(x^m)$ implies $x^m=(x^ma^m)^*x^m$. So $x^ma^m=(x^ma^m)^*x^ma^m$. Therefore $(x^ma^m)^*=x^ma^m$, together with $xa^{m+1}=a^m,~ax^2=x$, implies $a$ is *-DMP with index $m$ by Theorem~\ref{21}.
\end{proof}
\vspace{4mm}

It is known that $a^{D}$ exists if and only if $(a^k)^D$ exists for any positive integer $k$ if and only if $(a^k)^D$ exists for some positive integer $k$ \cite{D1958}. We find this property is  inherited by *-DMP.
\begin{thm}~Let $a\in S$ and $k$ a positive integer, then $a$ is *-DMP if and only if~$a^k$ is *-DMP.
\end{thm}
\begin{proof} Observe that $a^{D}$ exists and $aa^{D}$ is symmetric if and only if $(a^k)^D$ exists and $a^k(a^k)^{D}$ is symmetric. So $a$ is *-DMP if and only if~$a^k$ is *-DMP by Lemma~\ref{2}.
\end{proof}
\vspace{4mm}

Given two  *-DMP elements $a$ and $b$,  we may be of interest to consider conditions for the product $ab$ (resp. sum $a+b$) to be *-DMP.

\begin{thm}~~Let $a,~b\in S$ with $ab=ba,~ab^*=b^*a$. If both $a$ and $b$ are *-DMP, then $ab$ is *-DMP.
\end{thm}
\begin{proof}~Suppose that both $a$ and $b$ are *-DMP, then $a^{\scriptsize\textcircled{\tiny D}},~a^D~\text{and}~b^{\scriptsize\textcircled{\tiny D}},~b^D$ exist with $a^{\scriptsize\textcircled{\tiny D}}=a^D$, $b^{\scriptsize\textcircled{\tiny D}}=b^D$ by Lemma \ref{1}. Since $a^{\scriptsize\textcircled{\tiny D}}$ and $b^{\scriptsize\textcircled{\tiny D}}$ exist with $ab=ba,~ab^*=b^*a$, then $(ab)^{\scriptsize\textcircled{\tiny D}}$ exists with $(ab)^{\scriptsize\textcircled{\tiny D}}=a^{\scriptsize\textcircled{\tiny D}}b^{\scriptsize\textcircled{\tiny D}}$ (see~\cite[Theorem 4.3]{A2016}). Also, $(ab)^D$ exists with $(ab)^D=a^Db^D$ (see \cite[Lemma 2]{DD2012}).
So, $$(ab)^{\scriptsize\textcircled{\tiny D}}=a^{\scriptsize\textcircled{\tiny D}}b^{\scriptsize\textcircled{\tiny D}}=a^Db^D=(ab)^D.$$
Hence $ab$ is *-DMP by Lemma \ref{1}.
\end{proof}
\vspace{4mm}

\begin{thm}~Let $a,~b\in R$ with $ab=ba=0,~a^*b=0$. If both $a$ and $b$ are *-DMP, then $a+b$ is *-DMP.
\end{thm}
\begin{proof}~If both $a$ and $b$ are *-DMP, then $a^{\scriptsize\textcircled{\tiny D}},~a^D~\text{and}~b^{\scriptsize\textcircled{\tiny D}},~b^D$ exist with $a^{\scriptsize\textcircled{\tiny D}}=a^D$, $b^{\scriptsize\textcircled{\tiny D}}=b^D$ by Lemma \ref{1}. Since $a^{\scriptsize\textcircled{\tiny D}}$ and $b^{\scriptsize\textcircled{\tiny D}}$ exist with $ab=ba=0,~a^*b=0$, then $(a+b)^{\scriptsize\textcircled{\tiny D}}$ exists with $(a+b)^{\scriptsize\textcircled{\tiny D}}=a^{\scriptsize\textcircled{\tiny D}}+b^{\scriptsize\textcircled{\tiny D}}$ (see \cite[Theorem 4.4]{A2016}). Also, $(a+b)^D$ exists with $(a+b)^D=a^D+b^D$ (see \cite[Corollary 1]{D1958}).
So we have $$(a+b)^{\scriptsize\textcircled{\tiny D}}=a^{\scriptsize\textcircled{\tiny D}}+b^{\scriptsize\textcircled{\tiny D}}=a^D+b^D=(a+b)^D.$$
Hence $a+b$ is *-DMP by Lemma \ref{1}.
\end{proof}

\begin{eg}~The condition $ab=0,~a^{*}b=0~($without $ba=0)$ is not sufficient to show that $a+b$ is *-DMP, although both $a$ and $b$ are *-DMP.

Let $R=\mathbb{C}^{2\times 2}$ with transpose as involution,
$a=\begin{pmatrix}
i & 0\\
0 & 0
\end{pmatrix},~b=\begin{pmatrix}
                  0 & 0\\
                  -1 & 0
                  \end{pmatrix}$,
then $ab=a^{*}b=0$, but $ba\neq0$.
Since $a^{\scriptsize\textcircled{\tiny D}}=a^{\tiny\textcircled{\tiny \#}}=a^{\#}aa^{(1,3)}
                                     =\begin{pmatrix}
-i & 0\\
0 & 0
\end{pmatrix}=a^{\#}=a^{D}$, $a$ is *-DMP.
It is clear that $b$ is *-DMP.
Observe that  $a+b=\begin{pmatrix}
     i & 0\\
    -1 & 0
    \end{pmatrix}$, by calculation, we find that neither $a+b$ nor $(a+b)^2$ has any \{1,3\}-inverse. Since
$(a+b)^m=\begin{cases}
         (-1)^{\frac{m-1}{2}}(a+b)& m ~\text{is odd} \\
         (-1)^{\frac{m}{2}+1}(a+b)^2 & m ~\text{is even}
\end{cases}$, we conclude that
$(a+b)^m$ has no \{1,3\}-inverse for arbitrary positive integer $m$.
Hence $a+b$ is not *-DMP.
\end{eg}
\section{Pseudo core decomposition}
Core-nilpotent decomposition was introduced in \cite{D1974} for complex matrices. Later, P. Patr\'{i}cio and R. Puystjens \cite{D2004} generalized this decomposition from complex matrices to rings.
Let $a\in R$ with $a^{D_m}$ exists. The sum $a=c_a+n_a$ is called the core-nilpotent decomposition of $a$, where $c_a=aa^{D_m}a$ is the core part of $a$,~$n_a=(1-aa^{D_m})a$ is the nilpotent part of $a$. This decomposition is unique and it brings $n_a^m=0$, $c_an_a=n_ac_a=0$, $c_a^{\#}$ exists with $c_a^{\#}=a^{D_m}$.

Wang \cite{2016} introduced the core-EP decomposition for complex matrices. Let $A$ be a square complex matrix with index $m$, then $A=A_1+A_2$, where $A_1^{\#}$ exists, $A_2^m=0$ and $A_1^*A_2=A_2A_1=0$. In the following, we give the pseudo core decomposition of an element in $*$-rings.  Since the pseudo core inverse of a square complex matrix always exists and coincides with its core-EP inverse, then the pseudo core decomposition and core-EP decomposition of a square complex matrix coincide.
\begin{thm}\label{13}\emph{(Pseudo core decomposition)}~
Let $a\in R$ with $a^{\scriptsize\textcircled{\tiny D}_m}$ exists. Then  $a=a_1+a_2$, where\\
 $(1)$~$a_1^{\#}$ exists;\\
 $(2)$~$a_2^m=0$;\\
 $(3)$~$a_1^*a_2=a_2a_1=0$.
\end{thm}

\begin{proof}~Since $a^{\scriptsize\textcircled{\tiny D}_m}$ exists. Take $a_1=aa^{\scriptsize\textcircled{\tiny D}_m}a$ and $a_2=a-aa^{\scriptsize\textcircled{\tiny D}_m}a$, then $a_2^m=0$ and $a_1^*a_2=a_2a_1=0$. Next, we will prove that $a_1^{\#}$ exists.
In fact,
$$a_1=aa^{\scriptsize\textcircled{\tiny D}_m}a=(aa^{\scriptsize\textcircled{\tiny D}_m}a)^2(a^{\scriptsize\textcircled{\tiny D}_m})^2a\in a_1^2R~\text{and}~a_1=aa^{\scriptsize\textcircled{\tiny D}_m}a=a^{\scriptsize\textcircled{\tiny D}_m}(aa^{\scriptsize\textcircled{\tiny D}_m}a)^2\in Ra_1^2.$$
Hence $a_1^{\#}$ exists with $a_1^{\#}=(a^{\scriptsize\textcircled{\tiny D}_m})^2a$ (see \cite[Proposition 7]{D1976}).



\end{proof}

\begin{thm}~The pseudo core decomposition of an element in $R$ is unique.
\end{thm}
\begin{proof}~The proof is similar to \cite[Theorem 2.4]{2016}, the matrices case. We give the proof for completeness.

Let $a=a_1+a_2$ be the pseudo core decomposition of $a\in R$, where $a_1=aa^{\scriptsize\textcircled{\tiny D}_m}a$, $a_2=a-aa^{\scriptsize\textcircled{\tiny D}_m}a$.
Let $a=b_1+b_2$ be another pseudo core decomposition of $a$. Then $a^m=\sum\limits_{i=0}^{m}b_1^ib_2^{m-i}$.
Since $b_1^*b_2=0$ and $b_2^m=0$, then $(a^m)^*b_2=0$. Since $b_2b_1=0$, then $a^mb_1(b_1^m)^{\#}=b_1$. Therefore,
\begin{equation*}
\begin{aligned}
b_1-a_1&=b_1-aa^{\scriptsize\textcircled{\tiny D}_m}a=b_1-aa^{\scriptsize\textcircled{\tiny D}_m}b_1-aa^{\scriptsize\textcircled{\tiny D}_m}b_2=b_1-a^m(a^{\scriptsize\textcircled{\tiny D}_m})^mb_1-[a^m(a^{\scriptsize\textcircled{\tiny D}_m})^m]^*b_2\\
&=b_1-a^m(a^{\scriptsize\textcircled{\tiny D}_m})^ma^mb_1(b_1^m)^{\#}=b_1-a^mb_1(b_1^m)^{\#}=0.
\end{aligned}
\end{equation*}
Thus, $b_1=a_1$. Hence the pseudo core decomposition of $a$ is unique.
\end{proof}
\vspace{4mm}

Next, we exhibit some applications of the pseudo core decomposition. First, we give a characterization of the pseudo core inverse by using the pseudo core decomposition.
\begin{thm}\label{15}~Let $a\in R$ with $a^{\scriptsize\textcircled{\tiny D}_m}$ exists and let the pseudo core decomposition of $a$ be as in Theorem~\ref{13}. Then $a_1^{\tiny\textcircled{\tiny \#}}=a^{\scriptsize\textcircled{\tiny D}_m}$.
\end{thm}
\begin{proof}~Suppose $a^{\scriptsize\textcircled{\tiny D}_m}$ exists, then $a^{D_m}$ and $(a^m)^{(1,3)}$ exist by Lemma \ref{4} (1), as well as

$a^{\scriptsize\textcircled{\tiny D}_m}(a_1)^2
=a^{\scriptsize\textcircled{\tiny D}_m}(aa^{\scriptsize\textcircled{\tiny D}_m}a)^2
=aa^{\scriptsize\textcircled{\tiny D}_m}a
=a_1$;~
$a_1(a^{\scriptsize\textcircled{\tiny D}_m})^2
=aa^{\scriptsize\textcircled{\tiny D}_m}a(a^{\scriptsize\textcircled{\tiny D}_m})^2
=a^{\scriptsize\textcircled{\tiny D}_m}$;

$a_1a^{\scriptsize\textcircled{\tiny D}_m}
=aa^{\scriptsize\textcircled{\tiny D}_m}aa^{\scriptsize\textcircled{\tiny D}_m}
=aa^{\scriptsize\textcircled{\tiny D}_m}$, which implies
$(a_1a^{\scriptsize\textcircled{\tiny D}_m})^*=a_1a^{\scriptsize\textcircled{\tiny D}_m}$.\\
$\mbox{We~thus}\mbox{~get}~a_1^{\tiny\textcircled{\tiny \#}}=a^{\scriptsize\textcircled{\tiny D}_m}$.
\end{proof}
\vspace{4mm}

In the following, we use pseudo core decomposition to characterize *-DMP elements.
\begin{thm}
Let $a\in R$ with $a^{\scriptsize\textcircled{\tiny D}_m}$ exists and let the pseudo core decomposition of $a$ be as in Theorem~\ref{13}. Then the following are equivalent:\\
$(1)$~$a$ is *-DMP with index $m$;\\
$(2)$~$a_1$ is EP.
\end{thm}

\begin{proof}~$(1)\Leftrightarrow(2)$. $a$ is *-DMP with index $m$ if and only if $a^{\scriptsize\textcircled{\tiny D}_m}$ exists with $aa^{\scriptsize\textcircled{\tiny D}_m}=a^{\scriptsize\textcircled{\tiny D}_m}a$ by Theorem \ref{5} (1). According to Theorem \ref{15}, $a_1^{\tiny\textcircled{\tiny \#}}=a^{\scriptsize\textcircled{\tiny D}_m}$. By a simple calculation, $a_1a_1^{\tiny\textcircled{\tiny \#}}=aa_1^{\tiny\textcircled{\tiny \#}}=aa^{\scriptsize\textcircled{\tiny D}_m}$, and $a_1^{\tiny\textcircled{\tiny \#}}a_1=a_1^{\tiny\textcircled{\tiny \#}}a=a^{\scriptsize\textcircled{\tiny D}_m}a$.
So $aa^{\scriptsize\textcircled{\tiny D}_m}=a^{\scriptsize\textcircled{\tiny D}_m}a$ is equivalent to $a_1a_1^{\tiny\textcircled{\tiny \#}}=a_1^{\tiny\textcircled{\tiny \#}}a_1$, which is equivalent to, $a_1$ is EP (see \cite[Theorem 3.1]{DDD2014}).
\end{proof}

\begin{rem}~If $a$ is *-DMP with index $m$. Then the pseudo core decomposition of $a$ coincides with its
core-nilpotent decomposition. In fact, if $a$ is *-DMP with index $m$, then $a^{\scriptsize\textcircled{\tiny D}_m}=a^{D_m}$ by Lemma~$\ref{1}$. Hence the pseudo core decomposition and core-nilpotent decomposition coincide.
\end{rem}

\section{Pseudo core order}
In the following, $R^{\tiny\textcircled{\tiny \#}}$ and $R^{\scriptsize\textcircled{\tiny D}}$ denote the sets of all core invertible and pseudo core invertible elements in $R$, respectively.  $R^{\scriptsize\textcircled{\tiny D}_m}$ and $R_{\scriptsize\textcircled{\tiny D}_m}$ denote the sets of all pseudo core invertible and dual pseudo core invertible elements of index $m$, respectively.


Baksalary and Trenkler \cite{B2010} introduced the core partial order for
complex matrices of index one. Then, Raki\'{c} and Djordjevi\'{c} \cite{D2015} generalized the core partial order from complex matrices to $*$-rings. Let $a, b \in R^{\tiny\textcircled{\tiny \#}}$, the core partial order $a\stackrel{\tiny\textcircled{\tiny \#}}\leq b$ was defined as
$$a\stackrel{\tiny\textcircled{\tiny \#}}\leq b:~a^{\tiny\textcircled{\tiny \#}}a=a^{\tiny\textcircled{\tiny \#}}b~ \text{and}~aa^{\tiny\textcircled{\tiny \#}}=ba^{\tiny\textcircled{\tiny \#}}.$$
In \cite{2016}, Wang introduced the core-EP order for complex matrices.
Let $A, B\in \mathbb{C}^{n\times n}$,
the core-EP order $A\stackrel{\scriptsize\textcircled{\tiny \dag}}\leq B$  was defined as  $$A\stackrel{\scriptsize\textcircled{\tiny \dag}}\leq B:~ A^{\scriptsize\textcircled{\tiny \dag}}A=A^{\scriptsize\textcircled{\tiny \dag}}B~\text{and}~ AA^{\scriptsize\textcircled{\tiny \dag}}=BA^{\scriptsize\textcircled{\tiny \dag}},$$
where $A^{\scriptsize\textcircled{\tiny \dag}}$ denotes the core-EP inverse \cite{2014D} of $A$.

One can see \cite{D1978}, \cite{2010} for a deep study of partial order.

In what follows, we generalize the core-EP order from complex matrices to $*$-rings and give some properties.
\begin{defn}~Let $a, b\in R^{\scriptsize\textcircled{\tiny D}}$. The pseudo core order $a\stackrel{\scriptsize\textcircled{\tiny D}}\leq b$ is defined as
$\numberwithin{equation}{section}$\begin{equation}
a\stackrel{\scriptsize\textcircled{\tiny D}}\leq b:~a^{\scriptsize\textcircled{\tiny D}}a=a^{\scriptsize\textcircled{\tiny D}}b~\text{and}~aa^{\scriptsize\textcircled{\tiny D}}=ba^{\scriptsize\textcircled{\tiny D}}.
\end{equation}
\end{defn}

We extend some results of the core-EP order \cite{2016} from matrices to an arbitrary $*$-ring. First, we have the following result.
\begin{thm}~The pseudo core order is not a partial order but merely a pre-order.
\end{thm}
\begin{proof}~It is clear that the pseudo core order (4.1) is reflexive.
Let $a, b, c\in R^{\scriptsize\textcircled{\tiny D}}$,
$a\stackrel{\scriptsize\textcircled{\tiny D}}\leq b$ and $b\stackrel{\scriptsize\textcircled{\tiny D}}\leq c$.
Next, we prove $a\stackrel{\scriptsize\textcircled{\tiny D}}\leq c.$

Suppose $k=max\{I(a), I(b)\}$. From $aa^{\scriptsize\textcircled{\tiny D}}=ba^{\scriptsize\textcircled{\tiny D}}$ and $bb^{\scriptsize\textcircled{\tiny D}}=cb^{\scriptsize\textcircled{\tiny D}}$, it follows that
\begin{equation*}
\begin{aligned}
aa^{\scriptsize\textcircled{\tiny D}}&=ba^{\scriptsize\textcircled{\tiny D}}=ba(a^{\scriptsize\textcircled{\tiny D}})^2
=b^{2}(a^{\scriptsize\textcircled{\tiny D}})^2
=b^{k+1}(a^{\scriptsize\textcircled{\tiny D}})^{k+1}
=bb^{\scriptsize\textcircled{\tiny D}}b^{k+1}(a^{\scriptsize\textcircled{\tiny D}})^{k+1}
=cb^{\scriptsize\textcircled{\tiny D}}b^{k+1}(a^{\scriptsize\textcircled{\tiny D}})^{k+1}\\
&=cb^k(a^{\scriptsize\textcircled{\tiny D}})^{k+1}
=cb(a^{\scriptsize\textcircled{\tiny D}})^2=ca^{\scriptsize\textcircled{\tiny D}}.
\end{aligned}
\end{equation*}
Since $aa^{\scriptsize\textcircled{\tiny D}}=ba^{\scriptsize\textcircled{\tiny D}}$, then $a^{\scriptsize\textcircled{\tiny D}}
=a^{\scriptsize\textcircled{\tiny D}}(aa^{\scriptsize\textcircled{\tiny D}})^*
=a^{\scriptsize\textcircled{\tiny D}}(ba^{\scriptsize\textcircled{\tiny D}})^*
=a^{\scriptsize\textcircled{\tiny D}}[b^k(a^{\scriptsize\textcircled{\tiny D}})^k]^*
=a^{\scriptsize\textcircled{\tiny D}}
[bb^{\scriptsize\textcircled{\tiny D}}b^k(a^{\scriptsize\textcircled{\tiny D}})^k]^*
=a^{\scriptsize\textcircled{\tiny D}}
[b^k(a^{\scriptsize\textcircled{\tiny D}})^k]^*bb^{\scriptsize\textcircled{\tiny D}}$.
Equalities $a^{\scriptsize\textcircled{\tiny D}}a=a^{\scriptsize\textcircled{\tiny D}}b$, $b^{\scriptsize\textcircled{\tiny D}}b=b^{\scriptsize\textcircled{\tiny D}}c$ and
$a^{\scriptsize\textcircled{\tiny D}}=[b^k(a^{\scriptsize\textcircled{\tiny D}})^k]^*bb^{\scriptsize\textcircled{\tiny D}}$ yield that
$a^{\scriptsize\textcircled{\tiny D}}a=a^{\scriptsize\textcircled{\tiny D}}b
=[b^k(a^{\scriptsize\textcircled{\tiny D}})^k]^*bb^{\scriptsize\textcircled{\tiny D}}b
=[b^k(a^{\scriptsize\textcircled{\tiny D}})^k]^*bb^{\scriptsize\textcircled{\tiny D}}c
=a^{\scriptsize\textcircled{\tiny D}}c$.\\
We thus have $a\stackrel{\scriptsize\textcircled{\tiny D}}\leq c.$\\
However, the pseudo core order is not anti-symmetric (see \cite[Example 4.1]{2016}).
\end{proof}

\vspace{4mm}
The following result give some characterizations of the pseudo core order,
generalizing \cite[Theorem 4.2]{2016} from matrices to an arbitrary $*$-ring.
\begin{thm}\label{23}~Let $a, b\in R^{\scriptsize\textcircled{\tiny D}}$ with $k=$ max $\{I(a), I(b)\}$ and let the pseudo core decomposition of $a,~b$ be as in Theorem~\ref{13}. Then the following are equivalent:\\
$(1)$~$a\stackrel{\scriptsize\textcircled{\tiny D}}\leq b$;\\
$(2)$~$a^{k+1}=ba^k$ and $a^*a^k=b^*a^k$;\\
$(3)$~$a_1\stackrel{\tiny\textcircled{\tiny \#}}\leq b_1$.
\end{thm}
\begin{proof}~$(1)\Rightarrow(2)$.~Post-multiply $aa^{\scriptsize\textcircled{\tiny D}}=ba^{\scriptsize\textcircled{\tiny D}}$ by $a^{k+1}$, then we derive $a^{k+1}=ba^k$. From $a^{\scriptsize\textcircled{\tiny D}}a=a^{\scriptsize\textcircled{\tiny D}}b$, it follows that $a^*(a^{\scriptsize\textcircled{\tiny D}})^*=b^*(a^{\scriptsize\textcircled{\tiny D}})^*$. Post-multiply this equality by $a^*a^k$, then $a^*a^k=b^*a^k$.
\\
$(2)\Rightarrow(1)$. Equality $a^*a^k=b^*a^k$ yields that $(a^k)^*a=(a^k)^*b$. Pre-multiply this equality by $a^{\scriptsize\textcircled{\tiny D}}((a^{\scriptsize\textcircled{\tiny D}})^k)^*$, then $a^{\scriptsize\textcircled{\tiny D}}a=a^{\scriptsize\textcircled{\tiny D}}b$. Post-multiply $a^{k+1}=ba^k$ by $(a^{\scriptsize\textcircled{\tiny D}})^{k+1}$,
then $aa^{\scriptsize\textcircled{\tiny D}}=ba^{\scriptsize\textcircled{\tiny D}}$.\\
$(1)\Rightarrow(3)$.~From Theorem \ref{15} and $aa^{\scriptsize\textcircled{\tiny D}}=ba^{\scriptsize\textcircled{\tiny D}}$, it follows that
\begin{equation*}
\begin{aligned}
a_1a_1^{\tiny\textcircled{\tiny \#}}&=aa_1^{\tiny\textcircled{\tiny \#}}
=aa^{\scriptsize\textcircled{\tiny D}}
=ba^{\scriptsize\textcircled{\tiny D}}
=ba(a^{\scriptsize\textcircled{\tiny D}})^2=b^2(a^{\scriptsize\textcircled{\tiny D}})^2
=\cdots=b^k(a^{\scriptsize\textcircled{\tiny D}})^k
=bb^{\scriptsize\textcircled{\tiny D}}b^k(a^{\scriptsize\textcircled{\tiny D}})^k\\
&=bb^{\scriptsize\textcircled{\tiny D}}ba^{\scriptsize\textcircled{\tiny D}}
=b_1a_1^{\tiny\textcircled{\tiny \#}}.
\end{aligned}
\end{equation*}
Meanwhile, we have $aa^{\scriptsize\textcircled{\tiny D}}=aa^{\scriptsize\textcircled{\tiny D}}bb^{\scriptsize\textcircled{\tiny D}}$ by taking an involution on $aa^{\scriptsize\textcircled{\tiny D}}=bb^{\scriptsize\textcircled{\tiny D}}ba^{\scriptsize\textcircled{\tiny D}}=bb^{\scriptsize\textcircled{\tiny D}}aa^{\scriptsize\textcircled{\tiny D}}$.
So $a^{\scriptsize\textcircled{\tiny D}}=a^{\scriptsize\textcircled{\tiny D}}bb^{\scriptsize\textcircled{\tiny D}}$.
Therefore $a_1^{\tiny\textcircled{\tiny \#}}a_1
=a_1^{\tiny\textcircled{\tiny \#}}a=a^{\scriptsize\textcircled{\tiny D}}a
=a^{\scriptsize\textcircled{\tiny D}}b
=a^{\scriptsize\textcircled{\tiny D}}bb^{\scriptsize\textcircled{\tiny D}}b
=a_1^{\tiny\textcircled{\tiny \#}}b_1$.\\
$(3)\Rightarrow(1)$.~Since $aa^{\scriptsize\textcircled{\tiny D}}
=a_1a_1^{\tiny\textcircled{\tiny \#}}=b_1a_1^{\tiny\textcircled{\tiny \#}}
=bb^{\scriptsize\textcircled{\tiny D}}ba^{\scriptsize\textcircled{\tiny D}}$, then
\begin{equation*}
\begin{aligned}
aa^{\scriptsize\textcircled{\tiny D}}&=bb^{\scriptsize\textcircled{\tiny D}}baa^{\scriptsize\textcircled{\tiny D}}a^{\scriptsize\textcircled{\tiny D}}
=(bb^{\scriptsize\textcircled{\tiny D}}b)^2(a^{\scriptsize\textcircled{\tiny D}})^2
=bb^{\scriptsize\textcircled{\tiny D}}bb^k(b^{\scriptsize\textcircled{\tiny D}})^kb(a^{\scriptsize\textcircled{\tiny D}})^2
=b(bb^{\scriptsize\textcircled{\tiny D}}ba^{\scriptsize\textcircled{\tiny D}})a^{\scriptsize\textcircled{\tiny D}}=ba(a^{\scriptsize\textcircled{\tiny D}})^2\\
&=ba^{\scriptsize\textcircled{\tiny D}}.
\end{aligned}
\end{equation*}
Equalities $aa^{\scriptsize\textcircled{\tiny D}}=bb^{\scriptsize\textcircled{\tiny D}}ba^{\scriptsize\textcircled{\tiny D}}$ and $aa^{\scriptsize\textcircled{\tiny D}}=ba^{\scriptsize\textcircled{\tiny D}}$ yield that $aa^{\scriptsize\textcircled{\tiny D}}=aa^{\scriptsize\textcircled{\tiny D}}bb^{\scriptsize\textcircled{\tiny D}}$. Therefore $a^{\scriptsize\textcircled{\tiny D}}=a^{\scriptsize\textcircled{\tiny D}}bb^{\scriptsize\textcircled{\tiny D}}$. Hence $a^{\scriptsize\textcircled{\tiny D}}b=a^{\scriptsize\textcircled{\tiny D}}bb^{\scriptsize\textcircled{\tiny D}}b=a_1^{\tiny\textcircled{\tiny \#}}b_1=a_1^{\tiny\textcircled{\tiny \#}}a_1=a^{\scriptsize\textcircled{\tiny D}}a$.
\end{proof}

Wang and Chen \cite{W2015} gave some equivalences to $a \stackrel{\tiny\textcircled{\tiny \#}}\leq b$ under the assumption that $a$ is EP. Similarly, we give a characterization of $a \stackrel{\scriptsize\textcircled{\tiny D}}\leq b$ whenever $a$ is *-DMP. In the following result, $c_a$ and $c_b$  are the core part of the core-nilpotent decomposition of $a,~b$ respectively.

\begin{thm}~Let $a,~b\in R^{\scriptsize\textcircled{\tiny D}}$. If $a$ is *-DMP, then the following are equivalent:\\
$(1)$~$a \stackrel{\scriptsize\textcircled{\tiny D}}\leq b$;\\
$(2)$~$c_a\stackrel{\tiny\textcircled{\tiny \#}}\leq c_b$;\\
$(3)$~$a^{\scriptsize\textcircled{\tiny D}}b^{\scriptsize\textcircled{\tiny D}}=b^{\scriptsize\textcircled{\tiny D}}a^{\scriptsize\textcircled{\tiny D}}$ and $a^{\scriptsize\textcircled{\tiny D}}b=a^{\scriptsize\textcircled{\tiny D}}a$;\\
$(4)$~$a^{\scriptsize\textcircled{\tiny D}} \stackrel{\scriptsize\textcircled{\tiny D}}\leq b^{\scriptsize\textcircled{\tiny D}}$ and $a^{\scriptsize\textcircled{\tiny D}}b=a^{\scriptsize\textcircled{\tiny D}}a$.
\end{thm}

\begin{proof}~Let $k=$max\{$I(a),~I(b)$\}. If $a$ is *-DMP, then $a^{\scriptsize\textcircled{\tiny D}}=a^D$ by Lemma~\ref{1}~and $aa^{\scriptsize\textcircled{\tiny D}}=a^{\scriptsize\textcircled{\tiny D}}a$ by Theorem~\ref{5}.\\
$(1)\Rightarrow(2)$.~$a^{\scriptsize\textcircled{\tiny D}}=c_a^{\tiny\textcircled{\tiny \#}}$ (see \cite[Theorem 2.9]{A2016}) and $a^{\scriptsize\textcircled{\tiny D}}a=a^{\scriptsize\textcircled{\tiny D}}b$ imply $c_a^{\tiny\textcircled{\tiny \#}}a=c_a^{\tiny\textcircled{\tiny \#}}b$. From $a^{\scriptsize\textcircled{\tiny D}}b=a^{\scriptsize\textcircled{\tiny D}}a=aa^{\scriptsize\textcircled{\tiny D}}=ba^{\scriptsize\textcircled{\tiny D}}$, we have $a^{\scriptsize\textcircled{\tiny D}}b^D=b^Da^{\scriptsize\textcircled{\tiny D}}$. So, $a^{\scriptsize\textcircled{\tiny D}}bb^Db=bb^Dba^{\scriptsize\textcircled{\tiny D}}=bb^Db^k(a^{\scriptsize\textcircled{\tiny D}})^k=b^k(a^{\scriptsize\textcircled{\tiny D}})^k=aa^{\scriptsize\textcircled{\tiny D}}$. Therefore $c_a^{\tiny\textcircled{\tiny \#}}c_b=c_bc_a^{\tiny\textcircled{\tiny \#}}=c_ac_a^{\tiny\textcircled{\tiny \#}}=c_a^{\tiny\textcircled{\tiny \#}}c_a$.\\
$(2)\Rightarrow(1)$. $aa^{\scriptsize\textcircled{\tiny D}}=c_ac_a^{\tiny\textcircled{\tiny \#}}=c_bc_a^{\tiny\textcircled{\tiny \#}}=bb^Dba^{\scriptsize\textcircled{\tiny D}}=(bb^Db)^2(a^{\scriptsize\textcircled{\tiny D}})^2=b^2b^Db(a^{\scriptsize\textcircled{\tiny D}})^2=b(bb^Dba^{\scriptsize\textcircled{\tiny D}})a^{\scriptsize\textcircled{\tiny D}}=baa^{\scriptsize\textcircled{\tiny D}}a^{\scriptsize\textcircled{\tiny D}}=ba^{\scriptsize\textcircled{\tiny D}}$, and $a^{\scriptsize\textcircled{\tiny D}}a=c_a^{\tiny\textcircled{\tiny \#}}c_a=c_a^{\tiny\textcircled{\tiny \#}}c_b=a^{\scriptsize\textcircled{\tiny D}}bb^Db=a^{\scriptsize\textcircled{\tiny D}}a^{\scriptsize\textcircled{\tiny D}}(bb^Db)^2=a^{\scriptsize\textcircled{\tiny D}}a^{\scriptsize\textcircled{\tiny D}}ab=a^{\scriptsize\textcircled{\tiny D}}b$.\\
$(1)\Rightarrow(3)$. From $a^{\scriptsize\textcircled{\tiny D}}a=a^{\scriptsize\textcircled{\tiny D}}b$ and $aa^{\scriptsize\textcircled{\tiny D}}=ba^{\scriptsize\textcircled{\tiny D}}$, it follows that $$aa^{\scriptsize\textcircled{\tiny D}}b=aa^{\scriptsize\textcircled{\tiny D}}a=ba^{\scriptsize\textcircled{\tiny D}}a=baa^{\scriptsize\textcircled{\tiny D}},$$
which forces $aa^{\scriptsize\textcircled{\tiny D}}b^{\scriptsize\textcircled{\tiny D}}=b^{\scriptsize\textcircled{\tiny D}}aa^{\scriptsize\textcircled{\tiny D}}=b^{\scriptsize\textcircled{\tiny D}}b^{k+1}(a^{\scriptsize\textcircled{\tiny D}})^{k+1}=b^{k}(a^{\scriptsize\textcircled{\tiny D}})^{k+1}=a^{\scriptsize\textcircled{\tiny D}}$.
So $a^{\scriptsize\textcircled{\tiny D}}b^{\scriptsize\textcircled{\tiny D}}=(a^{\scriptsize\textcircled{\tiny D}})^2=b^{\scriptsize\textcircled{\tiny D}}a^{\scriptsize\textcircled{\tiny D}}.$\\
$(3)\Rightarrow(1)$. $ba^{\scriptsize\textcircled{\tiny D}}=b(a^{\scriptsize\textcircled{\tiny D}})^2a=b(a^{\scriptsize\textcircled{\tiny D}})^2b=b(a^{\scriptsize\textcircled{\tiny D}})^{k+1}b^k=b(a^{\scriptsize\textcircled{\tiny D}})^{k+1}b^{\scriptsize\textcircled{\tiny D}}b^{k+1}=bb^{\scriptsize\textcircled{\tiny D}}(a^{\scriptsize\textcircled{\tiny D}})^{k+1}b^{k+1}=bb^{\scriptsize\textcircled{\tiny D}}aa^{\scriptsize\textcircled{\tiny D}}$, together with $aa^{\scriptsize\textcircled{\tiny D}}=a^{\scriptsize\textcircled{\tiny D}}a=a^{\scriptsize\textcircled{\tiny D}}b=(a^{\scriptsize\textcircled{\tiny D}})^kb^k=(a^{\scriptsize\textcircled{\tiny D}})^kb^{\scriptsize\textcircled{\tiny D}}b^{k+1}=b^{\scriptsize\textcircled{\tiny D}}(a^{\scriptsize\textcircled{\tiny D}})^kb^{k+1}=bb^{\scriptsize\textcircled{\tiny D}}aa^{\scriptsize\textcircled{\tiny D}}$, implies $aa^{\scriptsize\textcircled{\tiny D}}=ba^{\scriptsize\textcircled{\tiny D}}$.\\
$(3)\Rightarrow(4)$. From $a^{\scriptsize\textcircled{\tiny D}}b^{\scriptsize\textcircled{\tiny D}}=b^{\scriptsize\textcircled{\tiny D}}a^{\scriptsize\textcircled{\tiny D}}$, it follows that (1) holds and
\begin{equation*}
\begin{aligned}
(a^{\scriptsize\textcircled{\tiny D}})^{\scriptsize\textcircled{\tiny D}}a^{\scriptsize\textcircled{\tiny D}}&=a^2(a^{\scriptsize\textcircled{\tiny D}})^{2}=a^2b^{k}(a^{\scriptsize\textcircled{\tiny D}})^{k+2}=a^2b^{\scriptsize\textcircled{\tiny D}}b^{k+1}(a^{\scriptsize\textcircled{\tiny D}})^{k+2}
=a^2b^{\scriptsize\textcircled{\tiny D}}a(a^{\scriptsize\textcircled{\tiny D}})^2\\
&=a^2b^{\scriptsize\textcircled{\tiny D}}a^{\scriptsize\textcircled{\tiny D}}=a^2a^{\scriptsize\textcircled{\tiny D}}b^{\scriptsize\textcircled{\tiny D}}=(a^{\scriptsize\textcircled{\tiny D}})^{\scriptsize\textcircled{\tiny D}}b^{\scriptsize\textcircled{\tiny D}},\\
a^{\scriptsize\textcircled{\tiny D}}(a^{\scriptsize\textcircled{\tiny D}})^{\scriptsize\textcircled{\tiny D}}&=a^{\scriptsize\textcircled{\tiny D}}a^2a^{\scriptsize\textcircled{\tiny D}}=aa^{\scriptsize\textcircled{\tiny D}}=b^{\scriptsize\textcircled{\tiny D}}a^2a^{\scriptsize\textcircled{\tiny D}}=b^{\scriptsize\textcircled{\tiny D}}(a^{\scriptsize\textcircled{\tiny D}})^{\scriptsize\textcircled{\tiny D}}.
\end{aligned}
\end{equation*}
$(4)\Rightarrow(3)$. Since $(a^{\scriptsize\textcircled{\tiny D}})^{\scriptsize\textcircled{\tiny D}}a^{\scriptsize\textcircled{\tiny D}}
=(a^{\scriptsize\textcircled{\tiny D}})^{\scriptsize\textcircled{\tiny D}}b^{\scriptsize\textcircled{\tiny D}}$ and
$a^{\scriptsize\textcircled{\tiny D}}(a^{\scriptsize\textcircled{\tiny D}})^{\scriptsize\textcircled{\tiny D}}
=b^{\scriptsize\textcircled{\tiny D}}(a^{\scriptsize\textcircled{\tiny D}})^{\scriptsize\textcircled{\tiny D}}$, then we obtain
$aa^{\scriptsize\textcircled{\tiny D}}
=a^2a^{\scriptsize\textcircled{\tiny D}}b^{\scriptsize\textcircled{\tiny D}}$ and
$aa^{\scriptsize\textcircled{\tiny D}}=b^{\scriptsize\textcircled{\tiny D}}a^2a^{\scriptsize\textcircled{\tiny D}}$.
So $b^{\scriptsize\textcircled{\tiny D}}a^{\scriptsize\textcircled{\tiny D}}
=(a^{\scriptsize\textcircled{\tiny D}})^2
=a^{\scriptsize\textcircled{\tiny D}}b^{\scriptsize\textcircled{\tiny D}}$.
\end{proof}
\vspace{4mm}

Wang and Chen \cite{W2015} proved that if $a \stackrel{*}\leq b$, $a^{\dag}$ exists, then $b^{\dag}$ exists if and only if $[b(1-aa^{\dag})]^{\dag}$ exists.
Similarly, we have the following result.
\begin{thm}~Let $a,~b\in R^{\scriptsize\textcircled{\tiny D}}$ with $a\stackrel{\scriptsize\textcircled{\tiny D}}\leq b$. Suppose $a$ is *-DMP. Then $b$ is *-DMP if and only if $b(1-aa^{\scriptsize\textcircled{\tiny D}})$ is *-DMP.
\end{thm}

\begin{proof}~From $a^{\scriptsize\textcircled{\tiny D}}a=a^{\scriptsize\textcircled{\tiny D}}b$ and $aa^{\scriptsize\textcircled{\tiny D}}=ba^{\scriptsize\textcircled{\tiny D}}$, it follows that $$aa^{\scriptsize\textcircled{\tiny D}}b=aa^{\scriptsize\textcircled{\tiny D}}a=ba^{\scriptsize\textcircled{\tiny D}}a=baa^{\scriptsize\textcircled{\tiny D}}.$$
Suppose that $b$ is *-DMP, then $bb^{\scriptsize\textcircled{\tiny D}}=b^{\scriptsize\textcircled{\tiny D}}b$. Next, we prove $[b(1-aa^{\scriptsize\textcircled{\tiny D}})]^{\scriptsize\textcircled{\tiny D}}=b^{\scriptsize\textcircled{\tiny D}}-a^{\scriptsize\textcircled{\tiny D}}$.
In fact, suppose $I(b)=k$, then
\begin{equation*}
\begin{aligned}
~~~~(b^{\scriptsize\textcircled{\tiny D}}-a^{\scriptsize\textcircled{\tiny D}})[b(1-aa^{\scriptsize\textcircled{\tiny D}})]^{k+1}
&=(b^{\scriptsize\textcircled{\tiny D}}-a^{\scriptsize\textcircled{\tiny D}})b^{k+1}(1-aa^{\scriptsize\textcircled{\tiny D}})=b^k(1-aa^{\scriptsize\textcircled{\tiny D}})-a^{\scriptsize\textcircled{\tiny D}}b^{k+1}(1-aa^{\scriptsize\textcircled{\tiny D}})\\
&=b^k(1-aa^{\scriptsize\textcircled{\tiny D}})=[b(1-aa^{\scriptsize\textcircled{\tiny D}})]^k;
\end{aligned}
\end{equation*}

$b(1-aa^{\scriptsize\textcircled{\tiny D}})(b^{\scriptsize\textcircled{\tiny D}}-a^{\scriptsize\textcircled{\tiny D}})=bb^{\scriptsize\textcircled{\tiny D}}-aa^{\scriptsize\textcircled{\tiny D}}$;

$b(1-aa^{\scriptsize\textcircled{\tiny D}})(b^{\scriptsize\textcircled{\tiny D}}-a^{\scriptsize\textcircled{\tiny D}})^2=(bb^{\scriptsize\textcircled{\tiny D}}-aa^{\scriptsize\textcircled{\tiny D}})(b^{\scriptsize\textcircled{\tiny D}}-a^{\scriptsize\textcircled{\tiny D}})=b^{\scriptsize\textcircled{\tiny D}}-b^{\scriptsize\textcircled{\tiny D}}aa^{\scriptsize\textcircled{\tiny D}}=b^{\scriptsize\textcircled{\tiny D}}-a^{\scriptsize\textcircled{\tiny D}}$.\\
We thus have $[b(1-aa^{\scriptsize\textcircled{\tiny D}})]^{\scriptsize\textcircled{\tiny D}}=b^{\scriptsize\textcircled{\tiny D}}-a^{\scriptsize\textcircled{\tiny D}}$. \\
So, $b(1-aa^{\scriptsize\textcircled{\tiny D}})[b(1-aa^{\scriptsize\textcircled{\tiny D}})]^{\scriptsize\textcircled{\tiny D}}=bb^{\scriptsize\textcircled{\tiny D}}-aa^{\scriptsize\textcircled{\tiny D}}$ and $[b(1-aa^{\scriptsize\textcircled{\tiny D}})]^{\scriptsize\textcircled{\tiny D}}b(1-aa^{\scriptsize\textcircled{\tiny D}})=b^{\scriptsize\textcircled{\tiny D}}b-b^{\scriptsize\textcircled{\tiny D}}baa^{\scriptsize\textcircled{\tiny D}}=bb^{\scriptsize\textcircled{\tiny D}}-aa^{\scriptsize\textcircled{\tiny D}}$.\\
Therefore, $b(1-aa^{\scriptsize\textcircled{\tiny D}})[b(1-aa^{\scriptsize\textcircled{\tiny D}})]^{\scriptsize\textcircled{\tiny D}}=[b(1-aa^{\scriptsize\textcircled{\tiny D}})]^{\scriptsize\textcircled{\tiny D}}b(1-aa^{\scriptsize\textcircled{\tiny D}})$. Hence $b(1-aa^{\scriptsize\textcircled{\tiny D}})$ is *-DMP.

Conversely, suppose that $b(1-aa^{\scriptsize\textcircled{\tiny D}})$ is *-DMP. Then,
  $[b(1-aa^{\scriptsize\textcircled{\tiny D}})]^{\scriptsize\textcircled{\tiny D}}=[b(1-aa^{\scriptsize\textcircled{\tiny D}})]^{D}$.
  We can easily check that $$(baa^{\scriptsize\textcircled{\tiny D}})^{\scriptsize\textcircled{\tiny D}}=(baa^{\scriptsize\textcircled{\tiny D}})^{\tiny\textcircled{\tiny \#}}=a^{\scriptsize\textcircled{\tiny D}}.$$
Since
$b=b(1-aa^{\scriptsize\textcircled{\tiny D}})+baa^{\scriptsize\textcircled{\tiny D}}$, $b(1-aa^{\scriptsize\textcircled{\tiny D}})baa^{\scriptsize\textcircled{\tiny D}}
=b(1-aa^{\scriptsize\textcircled{\tiny D}})aa^{\scriptsize\textcircled{\tiny D}}b=0$,
$baa^{\scriptsize\textcircled{\tiny D}}b(1-aa^{\scriptsize\textcircled{\tiny D}})=baa^{\scriptsize\textcircled{\tiny D}}(1-aa^{\scriptsize\textcircled{\tiny D}})b=0$,
and $(baa^{\scriptsize\textcircled{\tiny D}})^*b(1-aa^{\scriptsize\textcircled{\tiny D}})=b^*aa^{\scriptsize\textcircled{\tiny D}}(1-aa^{\scriptsize\textcircled{\tiny D}})b=0$,
then $b^{\scriptsize\textcircled{\tiny D}}=[b(1-aa^{\scriptsize\textcircled{\tiny D}})]^{\scriptsize\textcircled{\tiny D}}+a^{\scriptsize\textcircled{\tiny D}}$ (see~\cite[Theorem 4.4]{A2016}) and $b^{D}=[b(1-aa^{\scriptsize\textcircled{\tiny D}})]^{D}+(baa^{\scriptsize\textcircled{\tiny D}})^{\#}=[b(1-aa^{\scriptsize\textcircled{\tiny D}})]^{D}+a^{\scriptsize\textcircled{\tiny D}}$. Thus, $b$ is *-DMP.
\end{proof}

\section{Characterizations for $aa^{\scriptsize\textcircled{\tiny D}}=bb^{\scriptsize\textcircled{\tiny D}}$}
Let $a,~b\in R$. If $a^{\odot}$ and $b^{\odot}$ are some kind of generalized inverses of $a$ and $b$. It is very interesting to discuss whether $aa^{\odot}=bb^{\odot}$.
Koliha et al. \cite[Theorem 6.1]{D2002},  Mosi\'{c} et al. \cite[Theorem 3.7]{20144} and
Patr\'{i}cio et al. \cite[Theorem 2.3]{D2010} gave some equivalences for generalized Drazin inverses, image-kernel $(p, q)$-inverses and Moore-Penrose inverses, respectively. Here we give a characterization for $aa^{\scriptsize\textcircled{\tiny D}}=bb^{\scriptsize\textcircled{\tiny D}}$.

\begin{prop}\label{5.1}~Let $a,~b\in R^{\scriptsize\textcircled{\tiny D}}$. Then the following are equivalent:\\
$(1)~aa^{\scriptsize\textcircled{\tiny D}}=bb^{\scriptsize\textcircled{\tiny D}}aa^{\scriptsize\textcircled{\tiny D}}$;\\
$(2)~aa^{\scriptsize\textcircled{\tiny D}}=aa^{\scriptsize\textcircled{\tiny D}}bb^{\scriptsize\textcircled{\tiny D}}$;\\
$(3)~a^{\scriptsize\textcircled{\tiny D}}=a^{\scriptsize\textcircled{\tiny D}}bb^{\scriptsize\textcircled{\tiny D}}$;\\
$(4)~Ra^{\scriptsize\textcircled{\tiny D}}\subseteq Ra^{\scriptsize\textcircled{\tiny D}}bb^{\scriptsize\textcircled{\tiny D}}$.
\end{prop}

\begin{proof}~$(1)\Leftrightarrow(2)$ by taking an involution.\\
$(2)\Rightarrow(3)$. Pre-multiply $a^{\scriptsize\textcircled{\tiny D}}$ by $aa^{\scriptsize\textcircled{\tiny D}}
=aa^{\scriptsize\textcircled{\tiny D}}bb^{\scriptsize\textcircled{\tiny D}}$, then we get $a^{\scriptsize\textcircled{\tiny D}}=
a^{\scriptsize\textcircled{\tiny D}}bb^{\scriptsize\textcircled{\tiny D}}$.\\
$(3)\Rightarrow(4)$ is clear.\\
$(4)\Rightarrow(2)$. From $Ra^{\scriptsize\textcircled{\tiny D}}\subseteq Ra^{\scriptsize\textcircled{\tiny D}}bb^{\scriptsize\textcircled{\tiny D}}$, it follows that $a^{\scriptsize\textcircled{\tiny D}}
=xa^{\scriptsize\textcircled{\tiny D}}bb^{\scriptsize\textcircled{\tiny D}}$ for some $x\in R$.
Then, $aa^{\scriptsize\textcircled{\tiny D}}
=axa^{\scriptsize\textcircled{\tiny D}}bb^{\scriptsize\textcircled{\tiny D}}
=(axa^{\scriptsize\textcircled{\tiny D}}bb^{\scriptsize\textcircled{\tiny D}})
bb^{\scriptsize\textcircled{\tiny D}}
=aa^{\scriptsize\textcircled{\tiny D}}bb^{\scriptsize\textcircled{\tiny D}}$.
\end{proof}

The above proposition gives some equivalences to $aa^{\scriptsize\textcircled{\tiny D}}=bb^{\scriptsize\textcircled{\tiny D}}aa^{\scriptsize\textcircled{\tiny D}}$, which enrich the following result. $R^{-1}$ denotes all the invertible elements in $R$.

\begin{thm}\label{x}~Let $a,~b\in R^{\scriptsize\textcircled{\tiny D}}$ with $I(a)=m$. Then the following are equivalent:\\
$(1)~aa^{\scriptsize\textcircled{\tiny D}}=bb^{\scriptsize\textcircled{\tiny D}}$;\\
$(2)~aa^{\scriptsize\textcircled{\tiny D}}=aa^{\scriptsize\textcircled{\tiny D}}bb^{\scriptsize\textcircled{\tiny D}}$ and $u=aa^{\scriptsize\textcircled{\tiny D}}+1-bb^{\scriptsize\textcircled{\tiny D}}\in R^{-1}$;\\
$(3)~aa^{\scriptsize\textcircled{\tiny D}}=aa^{\scriptsize\textcircled{\tiny D}}bb^{\scriptsize\textcircled{\tiny D}}$ and
$v=a^m+1-bb^{\scriptsize\textcircled{\tiny D}}\in R^{-1}$;\\
$(4)~aa^{\scriptsize\textcircled{\tiny D}}$ commutes with $bb^{\scriptsize\textcircled{\tiny D}}$, $u=aa^{\scriptsize\textcircled{\tiny D}}+1-bb^{\scriptsize\textcircled{\tiny D}}\in R^{-1}$ and $s=bb^{\scriptsize\textcircled{\tiny D}}+1-aa^{\scriptsize\textcircled{\tiny D}}\in R^{-1}$;\\
$(5)~aa^{\scriptsize\textcircled{\tiny D}}$ commutes with $bb^{\scriptsize\textcircled{\tiny D}}$ and
$w=1-(aa^{\scriptsize\textcircled{\tiny D}}-bb^{\scriptsize\textcircled{\tiny D}})^2\in R^{-1}$;\\
$(6)~aa^{\scriptsize\textcircled{\tiny D}}$ commutes with $bb^{\scriptsize\textcircled{\tiny D}}$ and $b^{\scriptsize\textcircled{\tiny D}}aa^{\scriptsize\textcircled{\tiny D}}-a^{\scriptsize\textcircled{\tiny D}}bb^{\scriptsize\textcircled{\tiny D}}=b^{\scriptsize\textcircled{\tiny D}}-a^{\scriptsize\textcircled{\tiny D}}$.
\end{thm}

\begin{proof}~$(1)\Rightarrow(2)$-(6) is clear.\\
(2)$\Leftrightarrow$(3). Since $a^{\scriptsize\textcircled{\tiny D}_m}$ exists, then $a^{D_m}$ exists by Lemma \ref{4}. So $(a^m)^{\#}$ exists. Therefore $a^m+1-aa^{\scriptsize\textcircled{\tiny D}_m}\in R^{-1}$ (see \cite[Theorem 1]{D1997}). From $aa^{\scriptsize\textcircled{\tiny D}}
=aa^{\scriptsize\textcircled{\tiny D}}bb^{\scriptsize\textcircled{\tiny D}}$, it follows that
$aa^{\scriptsize\textcircled{\tiny D}}bb^{\scriptsize\textcircled{\tiny D}}
=bb^{\scriptsize\textcircled{\tiny D}}aa^{\scriptsize\textcircled{\tiny D}}
=aa^{\scriptsize\textcircled{\tiny D}}$ by Proposition \ref{5.1}.
Observe that
$(aa^{\scriptsize\textcircled{\tiny D}}+1-bb^{\scriptsize\textcircled{\tiny D}})
(a^m+1-aa^{\scriptsize\textcircled{\tiny D}})
=a^m+1-bb^{\scriptsize\textcircled{\tiny D}}$, and hence $u\in R^{-1}$ if and only if $v\in R^{-1}$.\\
(3)$\Rightarrow$(1).~Notice that
$aa^{\scriptsize\textcircled{\tiny D}}v=a^m+aa^{\scriptsize\textcircled{\tiny D}}-aa^{\scriptsize\textcircled{\tiny D}}bb^{\scriptsize\textcircled{\tiny D}}
=a^m$
and
$bb^{\scriptsize\textcircled{\tiny D}}v=bb^{\scriptsize\textcircled{\tiny D}}a^m
=bb^{\scriptsize\textcircled{\tiny D}}aa^{\scriptsize\textcircled{\tiny D}}a^m=aa^{\scriptsize\textcircled{\tiny D}}a^m=a^m$.
Therefore $aa^{\scriptsize\textcircled{\tiny D}}=bb^{\scriptsize\textcircled{\tiny D}}$.\\
(4)$\Rightarrow$(1). Since $ubb^{\scriptsize\textcircled{\tiny D}}
=aa^{\scriptsize\textcircled{\tiny D}}bb^{\scriptsize\textcircled{\tiny D}}
=uaa^{\scriptsize\textcircled{\tiny D}}bb^{\scriptsize\textcircled{\tiny D}}$,
 $saa^{\scriptsize\textcircled{\tiny D}}
=aa^{\scriptsize\textcircled{\tiny D}}bb^{\scriptsize\textcircled{\tiny D}} 
=saa^{\scriptsize\textcircled{\tiny D}}bb^{\scriptsize\textcircled{\tiny D}}$.
Hence $aa^{\scriptsize\textcircled{\tiny D}}
=aa^{\scriptsize\textcircled{\tiny D}}bb^{\scriptsize\textcircled{\tiny D}}
=bb^{\scriptsize\textcircled{\tiny D}}$.  \\
(5)$\Rightarrow$(4). Note that
$1-(aa^{\scriptsize\textcircled{\tiny D}}-bb^{\scriptsize\textcircled{\tiny D}})^2
=(bb^{\scriptsize\textcircled{\tiny D}}+1-aa^{\scriptsize\textcircled{\tiny D}})
(aa^{\scriptsize\textcircled{\tiny D}}+1-bb^{\scriptsize\textcircled{\tiny D}})
=(aa^{\scriptsize\textcircled{\tiny D}}+1-bb^{\scriptsize\textcircled{\tiny D}})
(bb^{\scriptsize\textcircled{\tiny D}}+1-aa^{\scriptsize\textcircled{\tiny D}})$.
Hence $w\in R^{-1}$ implies $u,~s\in R^{-1}$. \\
(6)$\Rightarrow$(1).
Post-multiply $b^{\scriptsize\textcircled{\tiny D}}aa^{\scriptsize\textcircled{\tiny D}}-a^{\scriptsize\textcircled{\tiny D}}bb^{\scriptsize\textcircled{\tiny D}}=b^{\scriptsize\textcircled{\tiny D}}-a^{\scriptsize\textcircled{\tiny D}}$ by $aa^{\scriptsize\textcircled{\tiny D}}$,
then $b^{\scriptsize\textcircled{\tiny D}}aa^{\scriptsize\textcircled{\tiny D}}-a^{\scriptsize\textcircled{\tiny D}}bb^{\scriptsize\textcircled{\tiny D}}aa^{\scriptsize\textcircled{\tiny D}}=b^{\scriptsize\textcircled{\tiny D}}aa^{\scriptsize\textcircled{\tiny D}}-a^{\scriptsize\textcircled{\tiny D}}$.
So, $a^{\scriptsize\textcircled{\tiny D}}=a^{\scriptsize\textcircled{\tiny D}}bb^{\scriptsize\textcircled{\tiny D}}aa^{\scriptsize\textcircled{\tiny D}}=a^{\scriptsize\textcircled{\tiny D}}bb^{\scriptsize\textcircled{\tiny D}}$.
Therefore, $b^{\scriptsize\textcircled{\tiny D}}
=b^{\scriptsize\textcircled{\tiny D}}aa^{\scriptsize\textcircled{\tiny D}}$.
Hence $aa^{\scriptsize\textcircled{\tiny D}}
=aa^{\scriptsize\textcircled{\tiny D}}bb^{\scriptsize\textcircled{\tiny D}}
=bb^{\scriptsize\textcircled{\tiny D}}aa^{\scriptsize\textcircled{\tiny D}}
=bb^{\scriptsize\textcircled{\tiny D}}$.
\end{proof}

Take $b=a^*$ in Theorem \ref{x}, then we obtain a characterization of *-DMP elements by applying Theorem \ref{22}.
\begin{cor}
Let $a\in R^{\scriptsize\textcircled{\tiny D}_m}\cap  R_{\scriptsize\textcircled{\tiny D}_m}$. Then the following are equivalent:\\
$(1)$ $a$ is *-DMP with index $m$;\\
$(2)~aa^{\scriptsize\textcircled{\tiny D}_m}=a_{\scriptsize\textcircled{\tiny D}_m}a$;\\
$(3)~aa^{\scriptsize\textcircled{\tiny D}_m}=aa^{\scriptsize\textcircled{\tiny D}_m}a_{\scriptsize\textcircled{\tiny D}_m}a$ and $u=aa^{\scriptsize\textcircled{\tiny D}_m}+1-a_{\scriptsize\textcircled{\tiny D}_m}a\in R^{-1}$;\\
$(4)~aa^{\scriptsize\textcircled{\tiny D}_m}=aa^{\scriptsize\textcircled{\tiny D}_m}a_{\scriptsize\textcircled{\tiny D}_m}a$ and $v=a^m+1-a_{\scriptsize\textcircled{\tiny D}_m}a\in R^{-1}$;\\
$(5)~aa^{\scriptsize\textcircled{\tiny D}_m}$ commutes with $a_{\scriptsize\textcircled{\tiny D}_m}a$, $u=aa^{\scriptsize\textcircled{\tiny D}_m}+1-a_{\scriptsize\textcircled{\tiny D}_m}a\in R^{-1}$ and $s=a_{\scriptsize\textcircled{\tiny D}_m}a+1-aa^{\scriptsize\textcircled{\tiny D}_m}\in R^{-1}$;\\
$(6)~aa^{\scriptsize\textcircled{\tiny D}_m}$ commutes with $a_{\scriptsize\textcircled{\tiny D}_m}a$ and
$w=1-(aa^{\scriptsize\textcircled{\tiny D}_m}-a_{\scriptsize\textcircled{\tiny D}_m}a)^2\in R^{-1}$;\\
$(7)~aa^{\scriptsize\textcircled{\tiny D}_m}$ commutes with $a_{\scriptsize\textcircled{\tiny D}_m}a$ and $a_{\scriptsize\textcircled{\tiny D}_m}^*aa^{\scriptsize\textcircled{\tiny D}_m}-a^{\scriptsize\textcircled{\tiny D}_m}a_{\scriptsize\textcircled{\tiny D}_m}a=a_{\scriptsize\textcircled{\tiny D}_m}^*-a^{\scriptsize\textcircled{\tiny D}_m}$.
\end{cor}

\vspace{0.8cm}

\noindent {\large\bf Acknowledgements}\\
This research is supported by the National Natural Science Foundation
of China (No.11371089), the Scientific Innovation Research of College Graduates in Jiangsu Province (No.KYZZ16$\_$0112),
the Natural Science Foundation of Jiangsu Province (No.BK20141327).

\end{document}